\documentclass[12pt, reqno]{amsart}
\usepackage{amssymb}
\usepackage{amsmath}
\usepackage{tikz-cd}
\usepackage{hyperref}
\usepackage[shortlabels]{enumitem}
\usepackage[capitalise]{cleveref}
\usepackage{fullpage}

\newtheorem{theorem}{Theorem}[section]
\newtheorem{lemma}[theorem]{Lemma}
\newtheorem{proposition}[theorem]{Proposition}
\newtheorem{corollary}[theorem]{Corollary}

\theoremstyle{definition}
\newtheorem{definition}[theorem]{Definition}
\theoremstyle{remark}
\newtheorem{remark}[theorem]{Remark}

\newcommand{\A}{\mathcal A}
\newcommand{\B}{\mathcal B}

\renewcommand{\L}{\mathcal L}
\newcommand{\M}{\mathcal M}

\renewcommand{\O}{\mathcal O}

\newcommand{\R}{\mathcal R}
\renewcommand{\S}{\mathcal S}

\renewcommand{\SS}{\mathbb S}

\newcommand{\iso}{\cong}
\newcommand{\cat}{\mathsf}
\newcommand{\union}{\cup}
\newcommand{\subsetof}{\subseteq}

\renewcommand{\:}{\colon}

\allowdisplaybreaks

\newcommand{\upset}{{\uparrow}}
\newcommand{\downset}{{\downarrow}}

\begin{document}

\title{The category of topological spaces and open maps does not have products}

\author{Guram Bezhanishvili}
\address{New Mexico State University}
\email{guram@nmsu.edu}

\author{Andre Kornell}
\address{Dalhousie University}
\email{akornell@dal.ca}

\thanks{Andre Kornell was supported by the Air Force Office of Scientific Research under Award No.~FA9550-21-1-0041.}


\begin{abstract}
    We prove that the category of topological spaces and open maps does not have binary products, thus resolving the Esakia problem in the negative. We also prove that the category of Kripke frames does not have binary products and that the category of complete Heyting algebras does not have binary coproducts. 
\end{abstract}

\maketitle

\section{Introduction}

It is well known that the category $\cat{Top}$ of topological spaces and continuous maps has products and that they are described by the standard topological product construction. 
The situation changes when we consider the category $\cat{TopOpen}$ of topological spaces and \emph{open maps},
which are the continuous maps such that the image of each open set is open. 
The existence of products in $\cat{TopOpen}$ was explicitly stated by Esakia as an open problem as early as the late 1980s
at the Esakia-Janelidze seminar at Tbilisi State University.
More recently, the problem was revisited on MathOverflow \cite{117174} with no clear resolution. We prove that $\cat{TopOpen}$ does not have products, thus resolving the Esakia problem in the negative.

Specifically, we prove that the product of the Sierpi\'{n}ski space $\SS$ with itself does not exist in $\cat{TopOpen}$. We argue by contradiction. We suppose that such a product does exist, and we prove that arbitrarily large posets embed into this product when it is equipped with the opposite of its specialization order.
We obtain these arbitrarily large posets as subsets of the von Neumann universe $V$, which we order by the reflexive transitive closure of the membership relation. 
Throughout, we view posets, and more generally, preordered sets, as topological spaces with the convention that \emph{downsets} are open.

We also prove
that the product of $\SS$ with itself does not exist in the full subcatgeories of $\cat{TopOpen}$ that consist of $T_0$-spaces and of sober spaces. If we view $\SS$ as the two-element poset with $0<1$, then its product with itself also fails to exist in the category $\cat{PreOpen}$ of preordered sets and open maps and in its full subcategories that consist of posets and of well-ordered posets.

In contrast, as was communicated to us by Peter Johnstone, binary products do exist in the wide subcategory of $\cat{TopOpen}$ whose morphisms are local homeomorphisms. In addition, arbitrary products exist in the full subcategory of $\cat{TopOpen}$ consisting of hyperstonean spaces. Indeed, this subcategory is contravariantly equivalent to the category of commutative von Neumann algebras and unital normal $*$-homomorphsims \cite[Thm.~1.1]{MR4310049}, which is known to have coproducts \cite[Sec.~8.2]{MR0201989}.

As a consequence, we prove that the coproduct of the three-element Heyting algebra with itself does not exist in the category $\cat{CHA}$ of complete Heyting algebras and complete Heyting algebra homomorphisms. From this, we derive De Jongh's theorem \cite{deJ80} that the free complete Heyting algebra on two-generators does not exist. We also observe that our main results are not a direct corollary of De Jongh's theorem by showing that the open set functor $\O$ from $\cat{TopOpen}$, as well as the subcategories that we consider in this paper, into $\cat{CHA}$ does not have a right adjoint. 

We conclude the paper with an application of our results to modal logic by showing that the category of Kripke frames does not have binary products. From this, we derive that the category of complete BAOs, i.e., Boolean algebras with an operator, does not have binary coproducts either.

\section{Hereditary antichains} \label{sec: antichains}

As usual, a preorder is a reflexive and transitive relation $\le$. For a subset $S$ of a preordered class $P$, we define ${\downarrow}S = \{ p \in P : p\le s \mbox{ for some } s\in S\}$, and we define ${\uparrow}S$ similarly. Then, $S$ is a {\em downset} if ${\downarrow}S=S$ and an {\em upset} if ${\uparrow}S =S$.

\begin{definition}
Let $P$ be a partially ordered class. An \emph{antichain} of $P$ is a set $x \subseteq P$ that satisfies $p_1 \not < p_2$ for all $p_1, p_2 \in x$. An antichain $x$ is \emph{trivial} if it is empty or a singleton; otherwise, it is \emph{nontrivial}. A set $M \subseteq P$ is \emph{convex} if $p \leq q \leq r$ with $p,r \in M$ implies that $q \in M$.
\end{definition}

\begin{definition}
Let $V$ be the class of hereditary sets, that is, the von Neumann universe.
We partially order $V$ by defining $x < y$ if $x$ is in the transitive closure of $y$. Thus, $x < y$ iff $x \in x_1 \in \cdots \in x_n \in y$ for some sets $x_1, \ldots, x_n$, where $n$ may be zero. For each convex set $M \subseteq V$, we define a cumulative hierarchy $S_\alpha(M)$ by transfinite recursion on the ordinal $\alpha$:
\begin{enumerate}
\item $S_0(M) = M$,
\item $S_{\alpha +1}(M) = S_\alpha(M) \cup \{x : x\text{ is a nontrivial antichain of } S_{\alpha}(M)\}$,
\item $S_\alpha(M) = \bigcup\{S_\beta(M) : \beta < \alpha\}$ for each limit ordinal $\alpha$.
\end{enumerate}
We define $S(M) = \bigcup\{S_\alpha(M): \alpha \text{ is an ordinal}\}$, and we refer to $S(M)$ as the class of \emph{hereditary nontrivial antichains over $M$}.
\end{definition}

Each element of $S(M)$ is either an element of $M$ or a nontrivial antichain of $S(M)$.
When $M$ is an antichain of $V$,
the class $S(M)$ is a subclass of $V(M)$, the class of hereditary sets over $M$ \cite[p.~250]{Jech2006}. 
The following claims are routine to verify.

\begin{lemma}\label{lem: auxiliary}
Let $\alpha$ be an ordinal.
\begin{enumerate}[label=\normalfont(\arabic*), ref = \thelemma(\arabic*)]
\item $S_\alpha(M)$ is a downset of $S(M)$ for each convex set $M \subseteq V$, \label[lemma]{lem: auxiliary 1}
\item $S_{\alpha+1}(M) \setminus S_\alpha(M)$ is an antichain of $S(M)$ for each convex set $M \subseteq V$, \label[lemma]{lem: auxiliary 2}
\item $S_\alpha(M) = S(M) \cap S_\alpha(M')$ for all antichains $M, M' \subseteq V$ such that $M \subseteq M'$, \label[lemma]{lem: auxiliary 3}
\item $S(M) \subseteq S(M')$ for all convex sets $M, M' \subseteq V$ such that $M \subseteq S_\alpha(M')$. \label[lemma]{lem: auxiliary 4}
\end{enumerate}
\end{lemma}

\begin{lemma}\label{antichain lemma}
Let $M$ be a set, and let $m' \not \in M$. If $M' = M \cup \{m'\}$ is an antichain of $V$, then
\begin{enumerate}[label=\normalfont(\arabic*), ref = \thelemma(\arabic*)]
\item $\{x, m'\}$ is a nontrivial antichain for all $x \in S(M)$,
\item $\{\{x, m'\}: x \in A\}$ is an antichain of $S(M')$ for all subsets $A \subseteq S(M)$.
\end{enumerate}
\end{lemma}

\begin{proof}
Assume that $M' = M \cup \{m'\}$ is an antichain.
We prove that $\{x, m'\}$ is a nontrivial antichain for all $x \in S_\alpha(M)$ by transfinite induction on $\alpha$. Assume that $\{x, m'\}$ is a nontrivial antichain for all $x \in S_\beta(M)$ and all $\beta < \alpha$. If $\alpha$ is zero, then $\{x, m'\}$ is a nontrivial antichain for all $x \in S_\alpha(M) = S_0(M) = M$ because $M'$ is an antichain and $m' \not \in M$. If $\alpha$ is limit, then $\{x, m'\}$ is a nontrivial antichain for all $x \in S_\alpha(M)$ by the induction hypothesis because $S_\alpha(M) = \bigcup\{S_\beta(M) : \beta < \alpha\}$.

Assume that $\alpha$ is successor, let $x \in S_\alpha(M)$, and suppose that $\{x, m'\}$ is not a nontrivial antichain. It follows that $x \leq m'$ or $m' < x$. In the former case, there exists $m \in M$ such that $m \leq x \leq m'$, contradicting that $M'$ is an antichain and $m' \not \in M$. In the latter case, $x \in M$, contradicting that $M'$ is an antichain, or there exists $y \in x \subseteq S_{\alpha-1}(M)$ such that $m' \leq y$, contradicting the induction hypothesis. Thus, $\{x, m'\}$ is a nontrivial antichain. Therefore, by induction on $\alpha$, $\{x, m'\}$ is a nontrivial antichain for all $x \in S_\alpha(M)$ and all ordinals $\alpha$. We have proved claim (1).

Let $A$ be a subset of $S(M)$, and suppose that $\{\{x, m'\} : x \in A\}$ is not an antichain. Then, $\{x, m'\} < \{y, m'\}$ for some $x, y \in A$, and hence, $\{x, m'\} \leq y$ or $\{x, m'\} \leq m'$. In the former case, $m' < y$, contradicting that $\{y, m'\}$ is an antichain, and in the latter case, $x < m'$, contradicting that $\{x, m'\}$ is an antichain. Therefore, $\{\{x, m'\} : x \in A\}$ is an antichain. We have proved claim (2).
\end{proof}

\begin{theorem}\label{cardinality theorem}
Let $M$ be an antichain of $V$. If $M$ has three or more elements, then $S(M)$ is a proper class.
\end{theorem}

\begin{proof}
First, we use transfinite induction on $\alpha$ to show that the antichain $S_{\alpha + 1}(M) \setminus S_\alpha(M)$ is infinite for all infinite antichains $M$ and all ordinals $\alpha$.
Assume that $S_{\beta +1}(M) \setminus S_\beta(M)$ is infinite for all infinite antichains $M$ and all ordinals $\beta < \alpha$. Fix an infinite antichain $M$. If $\alpha$ is zero, then $S_{\alpha}(M) = S_0(M) = M$, and all the infinite subsets of $M$ are elements of $S_{\alpha+1}(M) \setminus S_{\alpha}(M) = S_1(M) \setminus M$ because $M$ is an antichain. Thus, $S_{\alpha+1}(M) \setminus S_{\alpha}(M)$ is infinite. Similarly, if $\alpha$ is successor, then $S_{\alpha +1}(M) \setminus S_\alpha(M)$ is an infinite antichain that contains the infinite subsets of the infinite antichain $S_\alpha(M) \setminus S_{\alpha-1}(M)$ as elements.

Assume that $\alpha$ is limit. Let $M = M_1 \cup M_2$ be a decomposition of $M$ into two disjoint infinite subsets $M_1$ and $M_2$. For each $m \in M_2$, let $y_m = \{\{x, m\} : x \in S_\alpha(M_1)\}$.
We find that $y_m \in S(M_1 \union\{m\})$ by \cref{antichain lemma}, so $y_m$ is an antichain of $S(M)$. Since $\alpha$ is a limit ordinal, $\{x, m\} \in S_\alpha(M)$ for all $x \in S_\alpha(M_1)$, and hence, $y_m \in S_{\alpha +1}(M)$. If $y_m \in S_\alpha(M)$, then $y_m \in S_\beta(M)$ for some $\beta < \alpha$, and hence, $x \in S_\beta(M_1)$ for all $x \in S_\alpha(M_1)$, contradicting that $S_{\beta +1}(M_1) \setminus S_\beta(M_1)$ is infinite. Thus, $y_m \in S_{\alpha +1}(M) \setminus S_\alpha(M)$. Since $y_m \neq y_{m'}$ whenever $m \neq m'$, we conclude that $S_{\alpha +1}(M) \setminus S_\alpha (M)$ is infinite.

Therefore, by transfinite induction on $\alpha$, $S_{\alpha+1}(M) \setminus S_\alpha(M)$ is an infinite antichain of $S(M)$ for all infinite antichains $M$ and all ordinals $\alpha$. In particular, $S(M)$ is a proper class for all infinite antichains $M$.

Next, we prove that $S_\omega(M)$ is an infinite set for all antichains $M$ with at least three distinct elements.
Let $M$ be an antichain that contains distinct elements $m_1$, $m_2$, and $m_3$. By \cref{lem: auxiliary 2}, $S_{\alpha +1}(M) \setminus S_\alpha(M)$ is an antichain for all ordinals $\alpha$. We use induction on $\alpha$ to show that this antichain has at least three elements for all $\alpha < \omega$. If $\alpha$ is zero, then $S_{\alpha+1}(M) \setminus S_\alpha(M) = S_1 (M) \setminus M$ contains the antichains $\{m_1, m_2\}$, $\{m_2, m_3\}$, and $\{m_3, m_1\}$. Similarly, if $\alpha$ is successor, $S_{\alpha +1}(M) \setminus S_\alpha(M)$ 
contains at least three distinct doubleton subsets of $S_{\alpha}(M) \setminus S_{\alpha - 1}(M)$. By induction on $\alpha$, $S_{\alpha +1}(M) \setminus S_\alpha(M)$
has at least three elements for all $\alpha < \omega$. Therefore, $S_\omega(M)$ is infinite for all antichains $M$ with at least three elements.

We have shown that $S_\omega(M)$ is infinite for all antichains $M$ with three or more elements and that $S(M)$ is a proper class for all antichains $M$ with infinitely many elements. We combine these conclusions to prove the theorem.

Let $M$ be an antichain that contains distinct elements $m_1$, $m_2$, and $m_3$. Then, the antichain $S_1(M) \setminus M$ contains distinct elements $\{m_1, m_2\}$, $\{m_2, m_3\}$, $\{m_3, m_1\}$ and $\{m_1, m_2, m_3\}$. Let $M' = \{\{m_1, m_2\}, \{m_2, m_3\}, \{m_3, m_1\}\}$. Since $M'$ is an antichain with three or more elements, we know that $S_\omega(M')$ is an infinite set. Applying \cref{antichain lemma}, we find that the set $\{\{x, \{m_1, m_2, m_3\}\}: x \in S_\omega(M')\}$ is an infinite antichain of $S(M' \cup \{\{m_1, m_2, m_3\}\})$. Therefore, $S(\{\{x, \{m_1, m_2, m_3\}\}: x \in S_\omega(M')\})$ is a proper class that is contained in $S(M)$. We conclude that $S(M)$ is itself a proper class.
\end{proof}

\section{Products do not exist in $\cat{TopOpen}$} 

We view each preordered set as a topological space, whose open subsets are its \emph{downsets}. Preordered sets are then exactly the Alexandrov spaces up to homeomorphism, and among them, posets are exactly the Alexandrov $T_0$-spaces (see, e.g., \cite[p.~45]{Joh1982}). It is more common to define the open subsets to be the upsets, but the opposite convention is more convenient for the construction in \cref{sec: antichains}.

Let $\cat{PreOpen}$ be the category of preordered sets and open maps, let $\cat{PosOpen}$ be the full subcategory consisting of posets, and let $\cat{WellOpen}$ be the full subcategory consisting of well-founded posets. Recall that a poset $P$ is said to be \emph{well-founded} if it has no strictly decreasing sequences, that is, if its opposite is Noetherian. Our proof that these three categories do not have binary products relies on a key result, \cref{injective lemma}, which in turn requires the following characterization of open maps between preordered sets.

\begin{lemma}\label{lem: open maps}
Let $P$ and $Q$ be preordered sets, and let $f\: P \to Q$ be a function. The following are equivalent:
\begin{enumerate}[label=\normalfont(\arabic*), ref = \thelemma(\arabic*)]
\item $f$ is an open map,
\item $f[\downset p] = \downset f(p)$ for all $p\in P$,
\item $f^{-1}[\upset q]= \upset f^{-1}(q)$ for all $q \in Q$.
\end{enumerate}
\end{lemma}

\begin{proof}
Since closure in Alexandrov spaces is $\upset$, (1) is equivalent to $f^{-1}[\upset S]= \upset f^{-1}[S]$ for all $S \subseteq Q$ by \cite[pp.~98--100]{RS1963}.
Since $\upset$ commutes with unions, the latter condition is equivalent to (3). Thus, (1) is equivalent to (3). The equivalence between (2) and (3) follows from \cite[Prop.~1.4.12]{Esakia2019} (where the author works with the topologies of upsets rather than downsets, so the order in the statement is reversed everywhere).
\end{proof}

For the following lemma, we recall that a \emph{chain} of a partially ordered class $P$ is a set $x \subseteq P$ that satisfies $p_1 \leq p_2$ or $p_2 \leq p_1$ for all $p_1, p_2 \in x$.

\begin{lemma} \label{injective lemma}
Let $M$ be a convex subset of $V$, and let $P$ be a preordered set. Assume that $M \cap \downset m$ is a chain for each $m \in M$.
For all ordinals $\alpha$ and all open maps $f\: S_\alpha(M) \to P$, if $f$ is injective on $M$, then $f$ is injective on $S_\alpha(M)$.
\end{lemma}

\begin{proof}
We prove the theorem by transfinite induction on $\alpha$. For brevity, let $T_\alpha = S_\alpha(M)$. In particular, $T_0 = M$. Assume that for all $\beta < \alpha$ and all open maps $g\: T_\beta \to P$, if $g$ is injective on $T_0$, then $g$ is injective on $T_\beta$. If $\alpha$ is zero, then the conclusion of the theorem is tautological. If $\alpha$ is limit, then the conclusion of the theorem follows immediately because $T_\alpha = \bigcup\{T_\beta: \beta < \alpha\}$ and $T_\beta$ is a downset of $T_\alpha$ for all $\beta < \alpha$. The set $T_\beta$ is a downset of $T_\alpha$ by \cref{lem: auxiliary 1} because $M$ is convex. Only the successor case remains.

Assume that $\alpha$ is a successor ordinal. Let $f\: T_\alpha \to P$ be an open map, and assume that it is injective on $T_0$. By the induction hypothesis, $f$ is injective on $T_{\alpha -1}$. Let $x,y \in T_\alpha$ be such that $f(x) = f(y)$. Because $f$ is an open map, we have that 
\begin{equation}\label{eq: open}\tag{$\ast$}
f[\downset x]=\downset f(x) = \downset f(y) = f[\downset y]
\end{equation}
by \cref{lem: open maps}.
We observe that $\downset x \setminus \{x\}$ is the intersection of $S(M)$ and the transitive closure of $x$, and it is likewise for $\downset y \setminus \{y\}$.

Suppose that exactly one of $x$ or $y$ is in $T_{\alpha -1}$. Without loss of generality, $x \not \in T_{\alpha -1}$ and $y \in T_{\alpha -1}$. We apply (\ref{eq: open}). For all $x' \in \downset x\setminus \{x\}$, there exists an element $y' \in \downset y$ such that $f(x') = f(y')$ and, hence, $x' = y'$ because $x', y' \in T_{\alpha-1}$ and $f$ is injective on $T_{\alpha-1}$. It follows that $\downset x \setminus \{x\} \subseteq \downset y$, and in particular, $x \subsetof \downset y$. We find that $y \not \in M$ because $\downset y$ is a chain for all $y \in M$. If $y \not \in x$, then $x \subseteq \downset y \setminus \{y\} \subseteq T_\beta$ for some $\beta < \alpha -1$, so $x \in T_{\beta +1} \subseteq T_{\alpha -1}$, contradicting our assumption that $x \not \in T_{\alpha-1}$. If $y \in x$, then $x$ is not a nontrivial antichain of $T_{\alpha-1}$, contradicting that $x \in T_\alpha \setminus T_{\alpha-1}$. We conclude that, for all $x, y \in T_\alpha$ such that $f(x) = f(y)$, either $x , y \in T_{\alpha -1}$ or $x, y \not \in T_{\alpha -1}$.

If $x, y \in T_{\alpha -1}$, then $x = y$ because $f$ is injective on $T_{\alpha -1}$. Assume that $x, y \not \in T_{\alpha -1}$. We apply (\ref{eq: open}). For all $x' \in \downset x \setminus \{x\}$, there exists an element $y' \in \downset y$ such that $f(x') = f(y')$. Since $x' \in T_{\alpha-1}$, this equality implies that $y' \in T_{\alpha -1}$ by the conclusion of the preceding paragraph. 
It follows that $x' = y'$ because $f$ is an injection on $T_{\alpha-1}$. Thus, $x' = y' \in \downset y \setminus \{y\}$. Overall, we conclude that $\downset x \setminus \{x\} \subseteq \downset y \setminus \{y\}$, and symmetrically, $\downset y \setminus \{y\} \subseteq \downset x \setminus \{x\}$. Thus, $\downset x \setminus \{x\} = \downset y \setminus \{y\}$. This implies that $x = y$ because $x$ is the set of maximal elements of $\downset x \setminus \{x\}$, and it is likewise for $y$. Indeed, $x \subseteq T_{\alpha -1}$, and each element of $\downset x \setminus \{x\}$ is either an element of $x$ or strictly below an element of $x$. Therefore, $x = y$ whether $x, y \in T_{\alpha -1}$ or $x, y \not \in T_{\alpha -1}$, and the other cases are impossible.

We have proved that the induction hypothesis implies the conclusion of the theorem when $\alpha$ is a successor ordinal. Therefore, by transfinite induction on $\alpha$, every open map $f\: T_\alpha \to P$ is an injection if it is an injection on $T_0$.
\end{proof}

\begin{theorem} \label{thm: PreOpen does not have products}
The following categories do not have binary products:
\begin{enumerate}[label=\normalfont(\arabic*), ref = \thelemma(\arabic*)]
\item $\cat{PreOpen}$,
\item $\cat{PosOpen}$,
\item $\cat{WellOpen}$.
\end{enumerate}
\end{theorem}

\begin{proof}
Suppose that $\cat{PreOpen}$ has binary products, and let $\SS = \{0,1\}$ be ordered by $0 < 1$. Let the preordered set $P$, together with open maps $p_1,p_2\: P \to \SS$, be a product of $\SS$ with itself.
Let $0, 1, 2, \ldots$ be the finite ordinals. Let 
\begin{enumerate}
\item $m_0 = \{\{0\}\}$,
\item $m_1 = \{m_0, 1\}$,
\item $m_2 = \{m_0, 2\}$,
\item $m_3 = \{m_0, 3\}$,
\end{enumerate}
and let $M = \{m_0, m_1, m_2, m_3\}$. We observe that $m_0 < m_1, m_2, m_3$ and that $\{m_1, m_2, m_3\}$ is an antichain because $m_0$ is not in the transitive closure of any ordinal. For the same reason, we have that $M$ is a convex set, and it is obvious that $M \cap \downset m$ is a chain for each $m \in M$.

Let $\alpha$ be an ordinal. An initial segment of the poset $S_\alpha(M)$ is depicted in \cref{figure} when $\alpha$ is a limit ordinal.

\begin{figure}[h]
\footnotesize
$$
\begin{tikzcd}
\vdots
&
\vdots
&
\text{\reflectbox{$\ddots$}}
&
\text{\reflectbox{$\ddots$}}
&
\text{\reflectbox{$\ddots$}}
\\
\{\{m_1,m_2\},\{m_1,m_3\}\}
\arrow[dash]{d}
\arrow[dash]{dr}
&
\{\{m_1,m_2\},m_3\}
\arrow[dash]{dl}
\arrow[dash]{ddr}
&
\cdots
&
\cdots
&
\cdots
\\
\{m_1,m_2\}
\arrow[dash]{d}
\arrow[dash]{dr}
&
\{m_1,m_3\}
\arrow[dash]{dl}
\arrow[dash]{dr}
&
\{m_2,m_3\}
\arrow[dash]{d}
\arrow[dash]{dl}
&
\{m_1,m_2,m_3\}
\arrow[dash]{dl}
\arrow[dash]{dll}
\arrow[dash]{dlll}
\\
m_1
\arrow[dash]{d}
&
m_2
\arrow[dash]{dl}
&
m_3
\arrow[dash]{dll}
\\
m_0
&
&
&
\end{tikzcd}
$$
\normalsize
\caption{The poset $S_\alpha(\{m_0,m_1,m_2, m_3\})$.}\label{figure}
\end{figure}
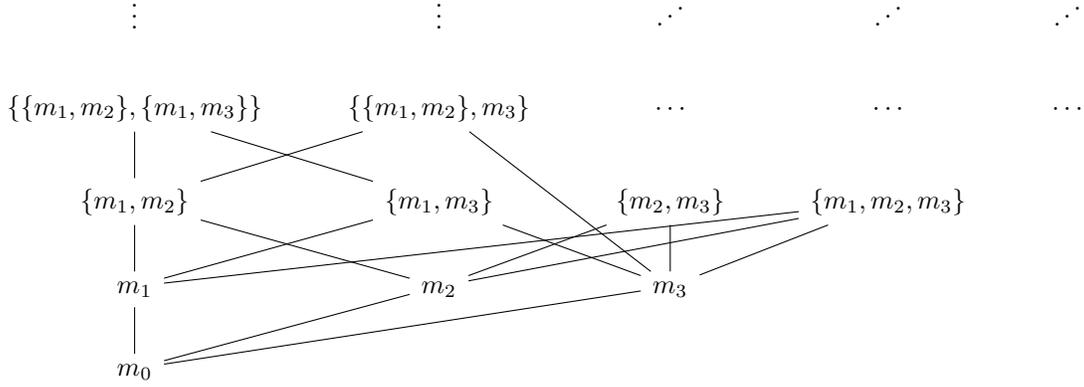

\newpage

\noindent For each $i \in\{1,2\}$, let $f_i\: S_\alpha(M) \to \SS$ be defined by
$$
f_i(x) = 
\begin{cases}
0 & \text{if }x = m_0, \\
0 & \text{if }x = m_i, \\
1 & \text{otherwise}.
\end{cases}
$$
The function $f_i$ is an open map by \cref{lem: open maps} because $f_i[\downset x] = \{0\} = \downset f_i(x)$
for $x \in \{m_0, m_i\}$ and $f_i[\downset x] = \{0,1\} = \downset f_i(x)$ for all other $x \in S_\alpha(M)$.
The universal property of the product yields an open map $f\:S_\alpha(M) \to P$ that makes the following diagram commute:
$$
\begin{tikzcd}
&
S_\alpha(M)
\arrow{ld}[swap]{f_1}
\arrow{rd}{f_2}
\arrow[dotted]{d}{f}
&
\\
\SS
&
P
\arrow{l}{p_1}
\arrow{r}[swap]{p_2}
&
\SS
\end{tikzcd}
$$

The function $m \mapsto (f_1(m), f_2(m))= (p_1(f(m)),p_2(f(m)))$ is an injection $M \to \SS \times \SS$, where the codomain is the usual product of posets. It follows that $f$ is injective on $M$ too. By \cref{injective lemma}, we conclude that $f$ is injective on $S_\alpha(M)$. Thus, the cardinality of $P$ is no smaller than the cardinality of $S_\alpha(M)$ for all ordinals $\alpha$, contradicting \cref{cardinality theorem} because $S_\alpha(\{m_1,m_2,m_3\}) \subseteq S_\alpha(M)$. Therefore, $\cat{PreOpen}$ does not have binary products.

We have proved claim (1). The same proof establishes claims (2) and (3) because $S_\alpha(M)$ is a well-founded poset for each ordinal $\alpha$.
\end{proof}

\begin{remark}
The construction in \cref{figure} is reminiscent of the coloring technique in modal logic (see, e.g., \cite{CZ1997,Nic06}). 
\end{remark}

We now prove that the category $\cat{TopOpen}$ of topological spaces and open maps, the full subcategory $\cat{TopOpen}_0$ of $T_0$-spaces, and the full subcategory $\cat{SobOpen}$ of sober spaces do not have binary products. A topological space is said to be \emph{sober} if each nonempty closed set that is not the union of two proper closed subsets is the closure of a unique point (see, e.g., \cite[p.~2]{PP2012}).
A poset is sober iff it is well-founded (see, e.g., \cite[Thm.~4.12]{EGP2007}).
Sobriety is a separation axiom that is natural to several areas of mathematics. For example, the Zariski spectrum of a commutative ring is sober but generally not Hausdorff \cite{Hoc1969}.

We apply \cref{thm: PreOpen does not have products} by using the familiar fact that right adjoints preserve products. Consequently, if a category has all binary products then its coreflective subcategories also have binary products.

\begin{proposition}\label{prop: coreflective}
We have that
\begin{enumerate}[label=\normalfont(\arabic*), ref = \thelemma(\arabic*)]
\item $\cat{PreOpen}$ is a coreflective subcategory of $\cat{TopOpen}$,
\item $\cat{PosOpen}$ is a coreflective subcategory of $\cat{TopOpen}_0$,
\item $\cat{WellOpen}$ is a coreflective subcategory of $\cat{SobOpen}$.
\end{enumerate}
\end{proposition}

\begin{proof}
Let $X$ be a topological space, and let $\A$ be the family of all open subsets that are Alexandrov in the subspace topology. The open set $\widehat X = \bigcup \A$ is also Alexandrov in the subspace topology because a space is Alexandrov iff each point has a least open neighborhood (see, e.g., \cite[sec.~II.1.8]{Joh1982}). Thus, $\widehat X$ is the largest open subset of $X$ that is Alexandrov in the subspace topology.

Let $P$ be a preordered set, and let $f\: P \to X$ be an open map. Then, the range of $f$ is open. 
It is also Alexandrov because $f[U_x]$ is the least open neighborhood of $f(x)$ whenever $U_x$ is the least open neighborhood of $x \in X$. Thus, $f[P]\in \A$, and $f$ factors uniquely through the inclusion $\widehat X \hookrightarrow X$:
$$
\begin{tikzcd}
P
\arrow[dotted]{d}[swap]{!}
\arrow{rd}{f}
&
\\
\widehat X
\arrow[hook]{r}
&
X
\end{tikzcd}
$$
We conclude that $\cat{PreOpen}$ is a coreflective subcategory of $\cat{TopOpen}$. 

We have proved claim (1). The same proof establishes claims (2) and (3). Indeed, each subspace of a $T_0$-space is $T_0$, and each open subspace of a sober space is sober (see, e.g, \cite[O-5.16(2)]{GHKLMS2003}).  
\end{proof}

\begin{remark}
    The category $\cat{Pre}$ of preordered sets and order-preserving maps is a coreflective subcategory of $\cat{Top}$, but the coreflector in the proof of \cref{prop: coreflective} is different. Indeed, the coreflector $\cat{Top}\to\cat{Pre}$ maps each topological space $X$ to the preordered set $(X,\le)$, where $\le$ is the dual of the specialization order. In other words, $x\le y$ iff $y\in\overline{\{x\}}$ (we work with the dual of the specialization order because of our earlier convention that downsets are open). 
    However, the Alexandrov topology of this preorder may be strictly finer than the topology of $X$, and hence the inclusion map may fail to be open. 
\end{remark}

We now complete the proof of our main result.

\begin{theorem}\label{thm: TopOpen does not have products}
The following categories do not have binary products:
\begin{enumerate}[label=\normalfont(\arabic*), ref = \thelemma(\arabic*)]
\item $\cat{TopOpen}$,
\item $\cat{TopOpen}_0$,
\item $\cat{SobOpen}$.
\end{enumerate}
\end{theorem}

\begin{proof}
Assume that $\cat{TopOpen}$ has binary products. The category $\cat{PreOpen}$ is a coreflective subcategory of $\cat{TopOpen}$ by \cref{prop: coreflective}, and thus, it too has binary products. This conclusion contradicts \cref{thm: PreOpen does not have products}. Therefore, $\cat{TopOpen}$ does not have binary products.

We have proved claim (1). The same proof establishes claims (2) and (3), with $\cat{PosOpen}$ in place of $\cat{PreOpen}$ in the case of claim (2) and $\cat{WellOpen}$ in place of $\cat{PreOpen}$ in the case of claim (3).
\end{proof}

\section{Coproducts do not exist in $\cat{CHA}$}

A \emph{Heyting algebra} is a bounded lattice $A$ such that $\wedge$ has a \emph{residual} $\to$ in the sense that 
\[
a\wedge x \le b \quad \Longleftrightarrow \quad x \le a \to b
\]
for all $a,b,x \in A$ \cite[p.~174]{BP74}. A \emph{complete Heyting algebra} is a Heyting algebra that is complete as a lattice. It is well known that a complete lattice is a Heyting algebra iff it satisfies the join-infinite distributive law \cite[p.~332]{PP2012}. Such complete lattices are known as \emph{frames} and have been studied extensively in pointfree topology \cite{Joh1982,PP2012}.  

A \emph{morphism} of complete Heyting algebras is a map $\varphi: A \to B$ that is a homomorphism of complete lattices such that $\varphi(a_1 \to a_2) = \varphi(a_1) \to \varphi(a_2)$ for all $a_1, a_2 \in A$. We prove that the resulting category $\cat{CHA}$ does not have coproducts. This also follows from a result of De Jongh \cite{deJ80}. However, our proof also shows that even finite Heyting algebras need not have a coproduct in this category. 

We begin by observing that $\cat{PosOpen}^{\mathrm{op}}$ is equivalent to a reflective subcategory of $\cat{CHA}$. For this we utilize De Jongh-Troelstra duality \cite{DT1966}. Let $\O\: \cat{PreOpen}^{\mathrm{op}} \to \cat{CHA}$ be the functor that maps each poset $P$ to the complete Heyting algebra $\O(P)$ and each open map $f\: P \to Q$ to the morphism $f^{-1}[-]\: \mathcal O (Q) \to \mathcal O (P)$. We use the same notation for the restriction of $\O$ to the full subcategories $\cat{PosOpen}^{\mathrm{op}}$ and $\cat{WellOpen}^{\mathrm{op}}$.

\begin{proposition}\label{prop: De Jongh Troelstra}
The functor $\mathcal O\:\cat{PosOpen}^{\mathrm{op}} \to \cat{CHA}$ is full, faithful, and has a left adjoint $\mathcal L\:\cat{CHA} \to \cat{PosOpen}^{\mathrm{op}}$. In particular, $\L(\O(P)) \cong P$ for each poset $P$.
\end{proposition}

\begin{proof}
The functor $\mathcal O$ is full and faithful by De Jongh-Troelstra duality; see \cref{lem: open maps} for equivalent characterizations of morphisms in $\cat{PosOpen}$. To show that it has a left adjoint, let $A$ be a complete Heyting algebra. Let $\S$ be the set of all $a \in A$ such that $\downset a \cong \O(P)$ for some poset $P$. By De Jongh-Troelstra duality, $a \in \S$ iff each $b \leq a$ is a join of completely join-irreducible elements. Let $s=\bigvee\S$. We show that 
$s \in \S$, which implies that $s$ is the largest element of $\S$.
Let $x\le s$. Since $A$ is a complete Heyting algebra, $A$ satisfies the join-infinite distributive law (see, e.g., \cite[Sec.~IV.7]{RS1963}), so $x=x\wedge s=\bigvee\{x \wedge y : y\in \S\}$. But each $x\wedge y$ is a join of completely join-irreducible elements because $y\in \S$. Thus, $x$ is a join of completely join-irreducible elements, and hence $s \in \S$. (In fact, $s=1$ iff $A\cong\O(P)$ for some poset $P$.)

Let $\L(A)$ be the poset of completely join-irreducible elements of $\downset s$. 
By De Jongh-Troelstra duality, there is an isomorphism $\varepsilon\:\downset s \cong \O(\L(A))$. For each $a\in A$, let $\rho_a :A \to \downset a$ be given by $\rho_a(b) = a \wedge b$, and define $\eta\: A \to \O(\L(A))$ to be the composition $\eta=\varepsilon\circ\rho_s$. 
Then $\eta$ is a composition of two morphisms of complete Heyting algebras and hence is a morphism of complete Heyting algebras. We show that $\eta$ has the desired universal property.

Let $Q$ be a poset, and let $\varphi: A \to \O(Q)$ be a morphism of complete Heyting algebras. 
Because complete homomorphic images are determined by principal filters and $A/\upset x \cong \downset x$ (see, e.g., \cite[Sec.~I.13]{RS1963}), we obtain that $\varphi$ factors through a morphism $\rho_a:\ A \to \downset a$ via an injective morphism $\iota: \downset a \to \O(Q)$. Since $\downset a$ is isomorphic to a complete subalgebra of $\O(Q)$, by De Jongh-Troelstra duality it is isomorphic to $\O(P)$ for some poset $P$, and hence $a\le s$. Thus, we obtain a morphism $\pi\: \O(\L(A)) \to \O(Q)$ such that $\pi \circ \eta = \varphi$:

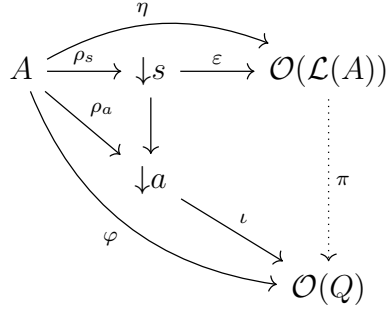
$$
\begin{tikzcd}
A
\arrow{r}{\rho_s}
\arrow{rd}{\rho_a}
\arrow[bend right = 30]{rrdd}[swap]{\varphi}
\arrow[bend left = 25]{rr}{\eta}
&
\downset s
\arrow{d}{}
\arrow{r}{\varepsilon}
&
\O(\L(A))
\arrow[dotted]{dd}{\pi}
\\
&
\downset a
\arrow{rd}{\iota}
&
\\
&
&
\O(Q)
\end{tikzcd}
$$

Because $\O$ is full, there exists an open map $f\: Q \to \L(A)$ such that $\O(f)=\pi$, and so $\O(f) \circ \eta = \varphi$.
Because $\O$ is faithful and $\eta$ is surjective, the open map $f$ is unique. Therefore, the assignment $A \mapsto \L(A)$ extends to a functor $\cat{CHA} \to \cat{PosOpen}^{\mathrm{op}}$ that is left adjoint to $\O$. Furthermore, $\L(\O(P)) \iso P$ for each poset $P$ because $\O$ is full and faithful (see, e.g., \cite[p.~88]{Mac71}).
\end{proof}

\begin{theorem}\label{thm: CHA does not have binary coproducts}
The category $\cat{CHA}$ does not have a coproduct of $\O(\SS)$ with itself.
\end{theorem}

\begin{proof}
Assume that the coproduct $\O(\SS) \sqcup \O(\SS)$ exists in $\cat{CHA}$. By \cref{prop: De Jongh Troelstra}, the functor $\mathcal L\:\cat{CHA} \to \cat{PosOpen}^{\mathrm{op}}$ is a left adjoint, so it preserves coproducts, and in particular, the coproduct $\SS \sqcup \SS \iso \L(\O(\SS)) \sqcup \L(\O(\SS)) \iso \L(\O(\SS) \sqcup \O(\SS))$ exists in $\cat{PosOpen}^{\mathrm{op}}$. We conclude that the product of $\SS$ with itself exists in $\cat{PosOpen}$. This conclusion contradicts the proof of \cref{thm: PreOpen does not have products}. Therefore, the coproduct $\O(\SS) \sqcup \O(\SS)$ does not exist in $\cat{CHA}$.
\end{proof}

\begin{corollary}[\cite{deJ80}]
    The free complete Heyting algebra on two generators does not exist.
\end{corollary}

\begin{proof}
Assume that there is a free complete Heyting algebra $A$ on two generators $p$ and $q$. Let $F \subseteq A$ be the principal filter $F = \upset (\neg \neg p \wedge \neg \neg q)$, and let $B = A / F$. We claim that $B$ is a coproduct of $\O(\SS)$ with itself in $\cat{CHA}$, where the injections $i,j\:\O(\SS) \to A$ are given by $i(\{0\}) = [p]$ and $j(\{0\}) = [q]$.
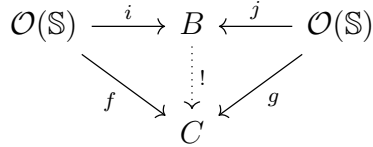
$$
\begin{tikzcd}
\O(\SS)
\arrow{r}{i}
\arrow{rd}[swap]{f}
&
B
\arrow[dotted]{d}{!}
&
\O(\SS)
\arrow{l}[swap]{j}
\arrow{ld}{g}
\\
&
C
&
\end{tikzcd}
$$

Let $C$ be a complete Heyting algebra, and let $f, g\: \O(\SS) \to C$ be morphisms of complete Heyting algebras. Let $h\: A \to C$ be the morphism that is defined by $h(p) = f (\{0\})$ and $h(q) = g(\{0\})$. We calculate that
\begin{align*}
h(\neg\neg p \wedge \neg \neg q & ) = \neg \neg h(p) \wedge \neg \neg h(q) = \neg \neg f(\{0\}) \wedge \neg \neg g(\{0\}) \\ & = f( \neg \neg \{0\}) \wedge g(\neg\neg \{0\})  = f(\SS) \wedge g(\SS) = 1 \wedge 1 = 1.
\end{align*}
Thus, $h[F] =\{1\}$, and hence, $h$ factors through the quotient morphism $A \to B$ via a morphism $\widetilde h\: B \to C$ (see, e.g., \cite[Sec.~I.13]{RS1963}).

We note that 
\[
\widetilde h(i(\{0\})) = \widetilde h [p] = h(p) = f(\{0\}),
\]
and similarly, 
\[
\widetilde h(j(\{0\})) = \widetilde h [q] = h(q) = g(\{0\}).
\]
Thus, $\widetilde h \circ i = f$, and $\widetilde h \circ j = g$. Furthermore, $\widetilde h$ is the unique morphism that satisfies these two equations because a morphism out of $B$ is uniquely determined by its values on $[p]$ and $[q]$. We conclude that $B$ is indeed a coproduct of $\O(\SS)$ with itself in $\cat{CHA}$. This conclusion directly contradicts \cref{thm: CHA does not have binary coproducts}. Therefore, there is no free complete Heyting algebra on two generators.
\end{proof}

\begin{remark}
De Jongh \cite{deJ80} proved that the free complete Heyting algebra on two generators does not exist. This result implies that the category $\cat{CHA}$ does not have binary coproducts because the free complete Heyting algebra on one generator does exist (see, e.g., \cite[Sec.~I.4.11]{Joh1982}) and De Jongh's theorem implies that this algebra does not have a coproduct with itself. \cref{thm: CHA does not have binary coproducts} is a  stronger result: it shows that even $\O(\SS)$ does not have a coproduct with itself. 
\end{remark}

It is well known that the free Heyting algebra $H$ on one generator is isomorphic to the Heyting algebra of upsets of the Rieger-Nishimura ladder $X$ (see, e.g., \cite[p.~204]{CZ1997}). Therefore, $H$ is isomorphic to $\O(Y)$, where $Y$ is the order-dual of $X$. Since the free complete Heyting algebra on one generator is obtained by adjoining a new top to $H$ (see, e.g., \cite[Sec.~I.4.11]{Joh1982}), it is isomorphic to $\O(Z)$, where $Z$ is obtained by adjoining a new top to $Y$. Furthermore, $Y$ is well-ordered because $X$ is dually well-ordered. Thus, the free complete Heyting algebra on one generator is isomorphic to $\O(Z)$ for a well-ordered poset $Z$. However, no part of \cref{thm: PreOpen does not have products,thm: TopOpen does not have products} is a direct corollary of De Jongh's theorem because the functor $\O$ does not have a right adjoint in any of these cases, as we now show. 

\begin{proposition}\label{prop: O does not have a right adjoint}
The following functors do not have right adjoints:
\begin{enumerate}[label=\normalfont(\arabic*), ref = \theproposition(\arabic*)]
\item $\mathcal O\:\cat{TopOpen}^{\mathrm{op}} \to \cat{CHA}$,
\item $\mathcal O\:\cat{TopOpen}_0^{\mathrm{op}} \to \cat{CHA}$,
\item $\mathcal O\:\cat{SobOpen}^{\mathrm{op}} \to \cat{CHA}$,
\item $\mathcal O\:\cat{PreOpen}^{\mathrm{op}} \to \cat{CHA}$,
\item $\mathcal O\:\cat{PosOpen}^{\mathrm{op}} \to \cat{CHA}$,
\item $\mathcal O\:\cat{WellOpen}^{\mathrm{op}} \to \cat{CHA}$.
\end{enumerate}
\end{proposition}

\begin{proof}
We prove (1); the proofs of (2)--(6) are identical.
Assume that $\O$ has a right adjoint $\R\: \cat{CHA} \to \cat{TopOpen}^{\mathrm{op}}$. For each natural number $n$, let $2^n$ be the $n$-fold Cartesian power of the set $2 =\{0,1\}$. Regarding $2^n$ as a discrete topological space, we have that $\O(2^n)$ is the powerset of $2^n$. Let $2^\omega$ be the Cantor space, let $\B(2^\omega)$ be its Borel $\sigma$-algebra, and let $\M \subseteq \B(2^\omega)$ be the set of meager Borel sets. Then, $\B(2^\omega)/\M$ is a complete atomless Boolean algebra; indeed, it is isomorphic to the regular open subsets of $2^\omega$ (see, e.g., \cite[Prop.~12.9]{Kop89}).

Let $f_n \: \O(2^n) \to \B(2^\omega)/\M$ be the morphism that maps each subset of $2^n$ to the corresponding clopen subset of $2^\omega$. Let $g_n\: \O(2^n) \to \O(\R(\B(2^\omega)/\M))$ be a morphism such that $\varepsilon \circ g_n = f_n$, where $\varepsilon$ is the counit for the adjunction $\O \dashv \R$ at $\B(2^\omega)/\M$. Thus, the following diagram commutes.
$$
\begin{tikzcd}
\O(\R(\B(2^\omega)/\M))
\arrow{r}{\varepsilon}
&
\B(2^\omega)/\M
\\
\O(2^n)
\arrow{ru}[swap]{f_n}
\arrow[dotted]{u}{g_n}
&
\end{tikzcd}
$$

Hence, the range of $\varepsilon$ contains $\bigcup_{n=1}^\infty f_n[\O(2^n)]$, which is a basis for the topology on~$2^\omega$. This basis generates $\B(2^\omega)/\M$ as a complete Boolean algebra, so $\B(2^\omega)/\M$ is a complete homomorphic image of the complete Heyting algebra $\O(\R(\B(2^\omega)/\M))$. Consequently, $\B(2^\omega)/\M \iso \downset U = \O(U)$ for some open subset $U \subseteq \R(\B(2^\omega)/\M)$ (see, e.g., \cite[Sec.~IV.8]{RS1963}).   
This yields that every open subset of $U$ is closed in $U$, which implies that $\O(U)$ is atomic. Thus, $\B(2^\omega)/\M$ is both atomic and atomless, in contradiction to the Baire category theorem. Therefore, the functor $\O$ does not have a right adjoint.
\end{proof}

\section{Kripke frames and \textsc{bao}s}

Posets and preordered sets are examples of Kripke frames, which play an important role in modal logic (see, e.g, \cite{BdRV2001}). We show that the category of Kripke frames and p-morphisms also lacks binary products. From this we derive that the category of complete \textsc{bao}s (as well as its full subcategory consisting of closure algebras) lacks binary coproducts. 

\begin{definition}
The category $\cat{KrF}$ of \emph{Kripke frames} and \emph{p-morphisms} is defined as follows:
\begin{enumerate}
\item an object is a set $X$ together with a binary relation $R$ on $X$,
\item a morphism from $(X,R)$ to $(Y,S)$ is a map $f\:X \to Y$ such that $f[R[x]] = S[f(x)]$ for all $x\in X$ or, equivalently, if $f^{-1}[S^{-1}(y)]=R^{-1}[f^{-1}(y)]$ for all $y \in Y$.
\end{enumerate}
\end{definition}

\begin{remark}\label{rem: PreOpen}
By \cref{lem: open maps}, $\cat{PreOpen}$ is equivalent to a full subcategory of $\cat{KrF}$. This equivalence maps each preorder $P$ to its opposite $P^{\mathrm{op}}$.
\end{remark}

\begin{theorem}\label{thm: KrF}
The category $\cat{PreOpen}$ is equivalent to a coreflective subcategory of $\cat{KrF}$. Therefore, $\cat{KrF}$ does not have binary products.
\end{theorem}

\begin{proof}
    Let $\mathfrak F = (X,R)$ be a Kripke frame. A subset $U$ of $X$ is an {\em $R$-upset} of $\mathfrak F$ provided $x\in U$ and $x\mathrel{R}y$ imply $y\in U$. 
    Let $\mathcal U$ be the set of all $R$-upsets $U$ of $\mathfrak F$ such that $(U,R_U)$ is a preorder, where $R_U := R \cap U^2$ is the restriction of $R$ to $U$. Set $Y=\bigcup\mathcal U$. Clearly $Y$ is an $R$-upset of $\mathfrak F$, and the restriction $R_Y$ is reflexive on $Y$. To see that it is transitive, let $x,y,z\in Y$ with $x\mathrel{R_Y}y\mathrel{R_Y}z$. Then $x\in U$ for some $U\in\mathcal U$, and hence also $y,z \in U$ since $U$ is an $R$-upset. It follows that $x\mathrel{R_Y}z$, since the restriction of $R$ to $U$ is transitive. Thus, $Y$ is the largest $R$-upset of $\mathfrak F$ that is a preorder.
    
    Let $\mathcal R (\mathfrak F)$ be $Y$ ordered by the converse of $R$, and let $\varepsilon\: \mathcal R(\mathfrak F)^{\mathrm{op}} \to \mathfrak F$ be the inclusion map, which is a p-morphism because $Y$ is an $R$-upset of $\mathfrak F$.
    Let $P$ be a preorder, and let $f: P^{\mathrm{op}} \to \mathfrak F$ be a p-morphism.
    Then $f[P]$ is an $R$-upset of $\mathfrak F$. Since $p\le q$ implies $f(q) \mathrel{R} f(p)$, the restriction of $R$ to $f[P]$ is reflexive. To see that it is transitive, let $p,q,r\in P$ with $f(p)\mathrel{R}f(q)\mathrel{R}f(r)$. Because $f$ is a p-morphism, there exist $q',r'\in P$ such that $r'\le q'\le p$, $f(q')=f(q)$, and $f(r')=f(r)$. Therefore, $r'\le p$, and so $f(p) \mathrel{R} f(r)$. Thus, $f[P]\in\mathcal U$, and hence $f[P]\subseteq Y$. We conclude that $f$ factors through $\varepsilon$ via $g^{\mathrm{op}}$, where $g\: P \to \mathcal R(\mathfrak F)$ is an open map by \cref{lem: open maps}.
    $$
    \begin{tikzcd}
    \mathcal R(\mathfrak F)^{\mathrm{op}}
    \arrow{r}{\varepsilon}
    &
    \mathfrak F
    \\
    P^{\mathrm{op}}
    \arrow{ru}[swap]{f}
    \arrow[dotted]{u}{g^{\mathrm {op}}}
    &
    \end{tikzcd}
    $$
    The open map $g$ is unique because $\varepsilon$ is an inclusion.
    Thus, we conclude that
    $\cat{PreOpen}$ is equivalent to a coreflective subcategory of $\cat{KrF}$. By \cref{thm: PreOpen does not have products}, $\cat{PreOpen}$ does not have binary products. Therefore, neither does $\cat{KrF}$.  
\end{proof}

A \emph{Boolean algebra with an operator} (\textsc{bao}) is a Boolean algebra $B$ that is equipped with a map $\Diamond\: B \to B$ that preserves finite joins (see \cite[Def.~2.13]{JT1951} and also \cite[Sec.~5.3]{BdRV2001}). A \textsc{bao} \emph{homomorphism} $(B_1, \Diamond_1) \to (B_2, \Diamond_2)$ is a Boolean algebra homomorphism $\varphi: B_1 \to B_2$ such that $\varphi(\Diamond_1 a) = \Diamond_2\varphi(a)$ for all $a \in B_1$. Furthermore, a {\em closure algebra} is a \textsc{bao} $(B,\Diamond)$ such that $a\le\Diamond a$ and $\Diamond\Diamond a\le\Diamond a$ for all $a \in B$ \cite[Def.~1.1]{MT1944}. Let $\cat{CBAO}$ be the category of complete \textsc{bao}s and complete \textsc{bao} homomorphsims, and let $\cat{CCA}$ be the full subcategory of $\cat{CBAO}$ consisting of closure algebras.

\begin{theorem}
The categories $\cat{CBAO}$ and $\cat{CCA}$ do not have binary coproducts.    
\end{theorem}

\begin{proof}
The proof of the two statements is similar. 
    We associate with each Kripke frame $\mathfrak F=(X,R)$ the \textsc{bao} $\wp(\mathfrak F) := (\wp(X),\Diamond_R)$, where $\wp(X)$ is the powerset of $X$ and 
    \[
    \Diamond_R U = R^{-1}[U] = \{x\in X:\ x \mathrel{R} u \mbox{ for some } u\in U\}.
    \]
    Also, with each p-morphism $f$ we associate the \textsc{bao} homomorphism $f^{-1}$. This defines a functor $\wp:\cat{KrF}^{\mathrm{op}}\to\cat{CBAO}$. By \cite[Thm.~3.13]{JT1951}, $R$ is a preorder iff the \textsc{bao} $(\wp(X),\Diamond_R)$ is a closure algebra. 
    Thus, $\wp$ restricts to a functor $\cat{PreOpen}^{\mathrm{op}}\to\cat{CCA}$ (see \cref{rem: PreOpen}). 
    
    By Thomason duality \cite{Tho1975}, $\wp$ is full and faithful. 
We show that it has a left adjoint $\L :\ \cat{CBAO} \to \cat{KrF}^{\mathrm{op}}$. Let $A$ be a complete \textsc{bao}. 
As is customary, let $\Box = \neg\Diamond\neg$. Then $\Box$ is a unary function on $A$ preserving finite meets. 
It is well known (see, e.g., \cite[Thm.~7.70]{CZ1997}) that if $a\in A$ satisfies $a\le \Box a$, then $A/\upset a$ is a homomorphic image of $A$. In this case, $A/\upset a$ is isomorphic to $\downset a$, where the corresponding operator on $\downset a$ is defined by $\Diamond_a b = a\wedge\Diamond b$ for each $b\le a$. 

Let $\S$ be the set of all $a \in A$ such that $a\le \Box a$ and $\downset a \cong \wp(\mathfrak F)$ for some Kripke frame $\mathfrak F$. By Thomason duality, 
the latter condition is equivalent to the following three conditions: $\downset a$ is complete and atomic and $\Diamond_a$ is completely additive (that is, $\Diamond_a\bigvee T=\bigvee\{\Diamond_a t:\ t\in T\}$). Let $s=\bigvee\S$. We show that $s \in \S$, which implies that $s$ is the largest element of $\S$. We have \[
s = \bigvee\S \le \bigvee\{\Box a :\ a\in \S\} \le \Box s.
\]
It is clear that $\downset s$ is complete. To see that it is atomic, let $0\ne b\le s$. Since $A$ is a complete Boolean algebra, $b=b\wedge s=\bigvee\{b \wedge a :\ a\in \S\}$, so there is $a\in\S$ with $b\wedge a\ne 0$. Because $\downset a$ is atomic, there is an atom under $b\wedge a$. Thus, $\downset s$ is atomic.

To show that $\Diamond_s$ is completely additive, we first observe that $\Box a \wedge \Diamond b \le \Diamond(a\wedge b)$ for all elements $a$ and $b$ in any \textsc{bao}. Indeed,
\begin{eqnarray*}
    & & \Box b \wedge \Diamond a \wedge \lnot\Diamond a = 0 \\ 
    & \Longrightarrow & 
    \Diamond a \wedge \Box b \wedge \Box\lnot a = 0  \\
    & \Longrightarrow &
    \Diamond a \wedge \Box (b \wedge \lnot a) = 0  \\ 
    & \Longrightarrow &
    \Diamond a \wedge \Box (b \wedge (\lnot a \vee \lnot b)) = 0 \\
    & \Longrightarrow &
    \Diamond a \wedge \Box b \wedge \Box \lnot (a \wedge b) = 0 \\ 
    & \Longrightarrow &
    \Diamond a \wedge \Box b \wedge \lnot \Diamond (a \wedge b) = 0 \\
    & \Longrightarrow &
    \Diamond a \wedge \Box b \le \Diamond (a \wedge b).
\end{eqnarray*}

Now, let $T\subseteq\downset s$, and let $x$ be an atom in $\downset s$ such that $x\le\Diamond_s \bigvee T = s\wedge\Diamond\bigvee T$. Since $x$ is an atom, $x\le a$ for some $a\in\S$. Therefore, by the above inequality,
\[
x \le a\wedge\Diamond \bigvee T \le \Box a\wedge\Diamond\bigvee T \le \Diamond (a \wedge \bigvee T) = \Diamond \bigvee \{ a\wedge t :\ t \in T \}.
\]
Thus, $x\le \Diamond_a \bigvee \{ a\wedge t :\ t \in T \}$, so $x \le \bigvee \{ \Diamond_a (a\wedge t) :\ t \in T\}$ since $\Diamond_a$ is completely additive. But then $x \le \bigvee \{ \Diamond_s t :\ t \in T\}$, and so $\Diamond_s$ is completely additive. 
Consequently, $s \in \S$. (In fact, $s=1$ iff $A\cong\wp(\mathfrak F)$ for some Kripke frame $\mathfrak F$.)

Hence, $\downset s \iso \wp(\L(A))$ for some Kripke frame $\mathfrak \L(A)$. Explicitly, $\mathfrak \L(A) = (X,R)$, where $X$ is the set of atoms of $\downset s$ and $x \mathrel{R} y$ iff $x\le\Diamond_s y$.  
Let $\varepsilon\:\downset s \to \wp(\L(A))$ be the isomorphism. Define $\eta\: A \to \wp(\L(A))$ as the composition $\eta = \varepsilon\circ\rho_s$, where we recall that $\rho_s(a)=s\wedge a$.  
Since $\eta$ is a composition of two morphisms of complete \textsc{bao}s, $\eta$ is a morphism of \textsc{bao}s. We show that $\eta$ has the desired universal property.

Let $\mathfrak F$ be a Kripke frame, and let $\varphi: A \to \wp(\mathfrak F)$ be a morphism of complete \textsc{bao}s. 
Because complete homomorphic images of \textsc{bao}s are determined by principal filters $\upset x$ such that $x\le\Box x$ and $A/\upset x \cong \downset x$, 
we obtain that $\varphi$ factors through the morphism $\rho_a:\ A \to \downset a$ via an injective morphism $\iota: \downset a \to \wp(\mathfrak F)$ with $a \leq \Box a$. Since $\downset a$ is isomorphic to a complete subalgebra of $\wp(\mathfrak F)$, by Thomason duality it is isomorphic to $\wp(\mathfrak G)$ for some Kripke frame $\mathfrak G$, and hence $a\le s$. Thus, we obtain a morphism $\pi\: \wp(\L(A)) \to \wp(\mathfrak F)$ such that $\pi \circ \eta = \varphi$.

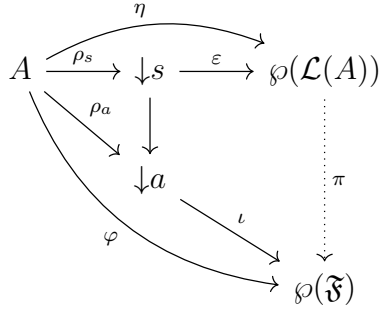
$$
\begin{tikzcd}
A
\arrow{r}{\rho_s}
\arrow{rd}{\rho_a}
\arrow[bend right = 30]{rrdd}[swap]{\varphi}
\arrow[bend left = 25]{rr}{\eta}
&
\downset s
\arrow{d}{}
\arrow{r}{\varepsilon}
&
\wp(\L(A))
\arrow[dotted]{dd}{\pi}
\\
&
\downset a
\arrow{rd}{\iota}
&
\\
&
&
\wp(\mathfrak F)
\end{tikzcd}
$$

Because $\wp$ is full, there exists a p-morphism $f\: \mathfrak F \to \L(A)$ such that $\wp(f)=\pi$, and so $\wp(f) \circ \eta = \varphi$.
Because $\wp$ is faithful and $\eta$ is surjective, the p-morphism $f$ is unique. Therefore, the assignment $A \to \L(A)$ extends to a functor $\cat{CBAO} \to \cat{KrF}^{\mathrm{op}}$ that is left adjoint to $\wp$. Furthermore, $\L(\wp(\mathfrak F)) \iso \mathfrak F$ for each Kripke frame $\mathfrak F$ because $\wp$ is full and faithful.

By \cref{thm: KrF}, there exist Kripke frames $\mathfrak F$ and $\mathfrak G$ whose product does not exist in $\cat{KrF}$. Assume that a binary coproduct $\wp(\mathfrak F) \sqcup \wp(\mathfrak G)$ exists in $\cat{CBAO}$. Since $\mathcal L$ is left adjoint to $\wp$, 
it preserves coproducts, and in particular, the coproduct 
\[
\mathfrak F \sqcup \mathfrak G \iso \L(\wp(\mathfrak F)) \sqcup \L(\wp(\mathfrak G)) \iso \L(\wp(\mathfrak F) \sqcup \wp(\mathfrak G))
\]
exists in $\cat{KrF}^{\mathrm{op}}$, so the product $\mathfrak F \times \mathfrak G$ exists in $\cat{KrF}$. This contradiction proves that the coproduct $\wp(\mathfrak F) \sqcup \wp(\mathfrak G)$ does not exist in $\cat{CBAO}$.
\end{proof}

\section*{Acknowledgements}

We would like to thank George Janelidze, Mamuka Jibladze, and Peter Johnstone for interesting discussions.

\newcommand{\etalchar}[1]{$^{#1}$}


\begin{thebibliography}{GHK{\etalchar{+}}03}

\bibitem[BD74]{BP74}
R.~Balbes and P.~Dwinger.
\newblock {\em Distributive lattices}.
\newblock University of Missouri Press, Columbia, MO, 1974.

\bibitem[BdRV01]{BdRV2001}
P.~Blackburn, M.~de~Rijke, and Y.~Venema.
\newblock {\em Modal logic}, volume~53 of {\em Cambridge Tracts in Theoretical
  Computer Science}.
\newblock Cambridge University Press, Cambridge, 2001.

\bibitem[Bez06]{Nic06}
N.~Bezhanishvili.
\newblock {\em Lattices of Intermediate and Cylindric Modal Logics}.
\newblock PhD thesis, ILLC, University of Amsterdam, 2006.

\bibitem[CZ97]{CZ1997}
A.~Chagrov and M.~Zakharyaschev.
\newblock {\em Modal logic}, volume~35 of {\em Oxford Logic Guides}.
\newblock The Clarendon Press, Oxford University Press, New York, 1997.

\bibitem[DJ80]{deJ80}
D.~H.~J. De~Jongh.
\newblock A class of intuitionistic connectives.
\newblock In {\em The {K}leene {S}ymposium ({P}roc. {S}ympos., {U}niv.
  {W}isconsin, {M}adison, {W}is., 1978)}, volume 101 of {\em Stud. Logic Found.
  Math.}, pages 103--111. North-Holland, Amsterdam-New York, 1980.

\bibitem[DJT66]{DT1966}
D.~H.~J. De~Jongh and A.~S. Troelstra.
\newblock On the connection of partially ordered sets with some
  pseudo-{B}oolean algebras.
\newblock {\em Indag. Math.}, 28:317--329, 1966.

\bibitem[EGP07]{EGP2007}
M.~Ern\'{e}, M.~Gehrke, and A.~Pultr.
\newblock Complete congruences on topologies and down-set lattices.
\newblock {\em Appl. Categ. Structures}, 15(1-2):163--184, 2007.

\bibitem[Eps13]{117174}
A.~Epstein.
\newblock Category of topological spaces with open or closed maps.
\newblock MathOverflow, 2013.
\newblock \texttt{URL:mathoverflow.net/q/117174} (version: 2013-01-08).

\bibitem[Esa19]{Esakia2019}
L.~Esakia.
\newblock {\em Heyting algebras}, volume~50 of {\em Trends in Logic---Studia
  Logica Library}.
\newblock Springer, Cham, 2019.
\newblock Translated from the original 1985 Russian edition by A. Evseev.

\bibitem[GHK{\etalchar{+}}03]{GHKLMS2003}
G.~Gierz, K.~H. Hofmann, K.~Keimel, J.~D. Lawson, M.~Mislove, and D.~S. Scott.
\newblock {\em Continuous lattices and domains}, volume~93 of {\em Encyclopedia
  of Mathematics and its Applications}.
\newblock Cambridge University Press, Cambridge, 2003.

\bibitem[Gui66]{MR0201989}
A.~Guichardet.
\newblock Sur la cat\'egorie des alg\`ebres de von {N}eumann.
\newblock {\em Bull. Sci. Math. (2)}, 90:41--64, 1966.

\bibitem[Hoc69]{Hoc1969}
M.~Hochster.
\newblock Prime ideal structure in commutative rings.
\newblock {\em Trans. Amer. Math. Soc.}, 142:43--60, 1969.

\bibitem[Jec06]{Jech2006}
T.~Jech.
\newblock {\em Set theory: The third millennium edition}.
\newblock Springer Monographs in Mathematics. Springer-Verlag, Berlin, 2006.

\bibitem[Joh82]{Joh1982}
P.~T. Johnstone.
\newblock {\em Stone spaces}, volume~3 of {\em Cambridge Studies in Advanced
  Mathematics}.
\newblock Cambridge University Press, Cambridge, 1982.

\bibitem[JT51]{JT1951}
B.~J\'{o}nsson and A.~Tarski.
\newblock Boolean algebras with operators. {I}.
\newblock {\em Amer. J. Math.}, 73:891--939, 1951.

\bibitem[Kop89]{Kop89}
S.~Koppelberg.
\newblock {\em Handbook of {B}oolean algebras. {V}ol. 1}.
\newblock North-Holland Publishing Co., Amsterdam, 1989.

\bibitem[Mac71]{Mac71}
S.~MacLane.
\newblock {\em Categories for the working mathematician}, volume Vol. 5 of {\em
  Graduate Texts in Mathematics}.
\newblock Springer-Verlag, New York-Berlin, 1971.

\bibitem[MT44]{MT1944}
J.~C.~C. McKinsey and A.~Tarski.
\newblock The algebra of topology.
\newblock {\em Ann. of Math. (2)}, 45:141--191, 1944.

\bibitem[Pav22]{MR4310049}
D.~Pavlov.
\newblock Gelfand-type duality for commutative von {N}eumann algebras.
\newblock {\em J. Pure Appl. Algebra}, 226(4):Paper No. 106884, 53, 2022.

\bibitem[PP12]{PP2012}
J.~Picado and A.~Pultr.
\newblock {\em Frames and locales}.
\newblock Frontiers in Mathematics. Birkh\"{a}user/Springer Basel AG, Basel,
  2012.

\bibitem[RS63]{RS1963}
H.~Rasiowa and R.~Sikorski.
\newblock {\em The mathematics of metamathematics.}
\newblock Pa\'{n}stwowe Wydawnictwo Naukowe, Warsaw, 1963.

\bibitem[Tho75]{Tho1975}
S.~K. Thomason.
\newblock Categories of frames for modal logic.
\newblock {\em J. Symbolic Logic}, 40(3):439--442, 1975.

\end{thebibliography}
\end{document}